\documentclass[12pt]{article}
\author{Anirban Das}
\title{Coupling two Markov Chains}

\usepackage{hyperref} 
\usepackage{makeidx} 
\usepackage[]{algorithm2e}
\usepackage{hyperref}
\usepackage{bbm}
\usepackage{amssymb}
\usepackage[english]{babel}
\usepackage[autostyle]{csquotes}
\usepackage{amsmath,amsfonts,amssymb,amsthm,epsfig,epstopdf,titling,url,array}
\usepackage{enumerate}
\makeindex
\usepackage[hypcap]{caption}
\usepackage[left=2cm,right=2cm,top=3cm,bottom=4cm]{geometry} 
\theoremstyle{plain}
\newtheorem{thm}{Theorem}[section]

\newtheorem*{remark}{Remark}
\newtheorem*{remarks}{Remarks}

\theoremstyle{definition}

\theoremstyle{remark}

\def\diy{\displaystyle}

\begin{document}
\def\cY{{\mathcal Y}}
\date{}
\title
{Constructions of  Markov processes in random environments which lead to a product form of the stationary measure}
\author{Anirban Das, \\
Department of Mathematics, Pennsylvania State University
}

\maketitle

\begin{abstract}{Recently Belopolskaya and Suhov (2015) studied Markov processes in a random environment, where the environment changes in a Markovian manner. They introduced constructions allowing the process to \enquote{interact with an environment}. This was done in such a manner that the combined process has the product of the stationary measures of the individual processes as it's stationary measure . In this paper, a new construction is implemented, related to a product form for the stationary measure. This construction can be carried out with almost no conditions. However it requires the use of an additional state, denoted by $c$. The extent to which the combined process uses state $c$ indicates how far this process is from naturally having a product form for the stationary measure. To specify various aspects of the construction, we use an example from queuing theory, which is studied in detail. We observe how our construction works in this example, especially how the combined process uses state $c$. Physical interpretations lead to a refined construction also administered on a queuing theory background. In the refined construction the use of state $c$ agrees with an intuition.  }
\end{abstract}

{\tiny AMS Classification (2010):  60J27 (Primary), 	60J28 (Secondary) .  }

\begin{section}{Introduction}\label{Introduction}
This paper concerns with dynamics of Markov chains in random environments. There
exists a substantial literature (see, e.g., Ref. \cite{zeitounilecture} and the bibliography therein) where such processes are considered in an environment that is randomly chosen but kept fixed throughout the time dynamics. In Ref. 
\cite{economou2005generalized}, a particular construction was proposed where the environment influences the basic process, but remains unaffected by it, resulting in a product-form of a stationary distribution. In \cite{belopolskaya2015models} a `combined' Markov process has been introduced, with basic and environment processes
influencing each other and the product-form of the stationary distribution still preserved. In the current paper we give a different construction of Markov models similar to \cite{belopolskaya2015models}. Our construction can be applied under more general conditions but with a added proviso: 
we have to affiliate an additional state for the combined process. 

\par Consider sets $Z= \{ z_1,z_2, \ldots z_M\}$ and  $X= \{ x_1,x_2, \ldots x_N\}$. We call $Z$ the environment space, and  $X$ the collection of basic states. We also are given a family of linear dissipative $N \times N$ matrices $\{Q^{z} \in \mathbbm{C}_{b}(X): z \in Z\}$  indexed by $z \in Z$, with entries  $Q^{z}(x_j|x_i)$ specifying the jump intensity from $x_i$ to $x_j$. Dissipativity means $Q^{z}(x_j|x_i)= (-1)^{\mathbbm{1}_{j}(i) }|Q^{z}(x_j|x_i)|$, and  $\sum\limits_{x' \in X} Q^{z}(x'|x) =0$. Thus, for a fixed $z \in Z$, matrix $Q^{z}$ yields the generator of a continuous-time Markov chain with values in $X$. It describes the dynamics of the basic state when the environment is fixed (see \cite{bremaud2013markov}, also \cite{ethier2009markov}). We call it a basic 
Markov chain or a basic process.\\

 We assume that the basic  process has a stationary measure $m^z(x), \; x \in X$, depending on $z \in Z$, with $m^z(x)> 0, \; \forall \; x,\; z$. Formally, $\forall \; x' \in X$
\begin{equation}{\label{1_1}}
\sum\limits_{x \in X} Q^{z}(x'|x)m^z (x)=0 \mbox{.}
\end{equation}

Next suppose we have $\{A^{x} \in \mathbbm{C}_{b}(Z): x \in X\}$ a  family of $M \times M$ dissipative matrices indexed by $x \in X$, with elements $A^{x}(z_j|z_i)$ representing  jump rates from $z_i$ to $z_j$. As earlier, we have  $A^{x}(z_j|z_i)= (-1)^{\mathbbm{1}_{j}(i) }|A^{x}(z_j|z_i)|$, and $\sum\limits_{z' \in Z} A^{x}(z'|z) =0$. Thus we have a family of Markov chains taking values in $Z$ and describing a random evolution of the environment for a fixed basic state. We suppose that each of these processes has a stationary measure $\nu^{x}(z) > 0$. That is, 
\begin{equation}{\label{1_2}}
\sum\limits_{z \in Z} A^{x}(z'|z) \nu^{x}(z)=0 \mbox{.}
\end{equation}
\par With these at hand, we would like to construct a continuous-time Markov chain on $X \times Z$ (referred to as a combined Markov chain or process) in a meaningful way, so that both the state and the environment can change together, while we still can use a product-measures as stationary measure for the combined process. This means the combined chain will have $g(x,z) = m^z(x) \times \nu^{x}(z)$ a stationary measure. Thus we are working with a random environment and a random basic process, interacting with each other. An 
advantage of preserving the product-form for the invariant measure has been discussed in \cite{gannon2014random},  while studying Jackson networks (see \cite{jackson1957networks}). Ref \cite{economou2005generalized} also puts forward a  significance of having a product-form for the stationary distribution of the combined process. 

\par Constructions leading to the above properties have been carried out in \cite{belopolskaya2015models}. These constructions are based upon certain assumptions. First, Ref. \cite{belopolskaya2015models} assumes that $\nu^{x}(z)$ is the same for all $x$. Next, it is assumed that the combined process can jump from state $(x,z)$ to a $(x',z')$ only when either $x=x'$ or $z=z'$. The construction presented here will not have these restrictions. However, as was said, we have to introduce a `transitional' state $c$, representing an additional level the combined process may attain. We show for any $\epsilon >0$ we can construct a combined process so that that $g(c)= \epsilon$.  By assuming that $m$ and $\nu$ are strictly positive we ensure $g$ is also positive valued.

\par Section 2.1 outlines general features of our construction. Section 2.2 gives a specific construction with a minimality property. Section 3 concerns with studying how this construction works out for a specific example. Section 3.1 describes an example from queuing theory, 3.2 analyses it in detail. Section 4.1 is dedicated to a refined construction developing the example in Section 3. Section 4.2 analyses the combined process for this refined construction. In the beginning of Section 4 we sketch a physical intuition behind the theorems of Section 4.2.
\end{section}

\begin{section}{A construction with a single transitional state}
\begin{subsection}{A general result}\label{general_construction}
Let $X$, $Z$ , $A^{x}$ and $ Q^{z}$ be as in the introduction. Define the state space $\cY= \{X \times Z\} \cup \{ c \}$. We call $X \times Z$ a natural space, and $c$ a transition state. 
Given $\epsilon >0$, define the positive function $g$ as
\begin{eqnarray}\label{eq:invmesg}
g(p) &=& m^z(x) \times \nu^{x}(z) \nonumber\\ 
     &\;&\quad\mbox{ when $p = (x,z)$, i.e., $p$ is a part of the natural space} \nonumber \\
 &=& \epsilon\;, \; \; \mbox{when $p=c$}.
\end{eqnarray}
Given a function $(x,z) \in X \times Z \mapsto \tau(x,z)$ define a matrix $R$ with entries $R (x',z'|x,z)$
\begin{equation}{\label{2_1}}
R (x',z'|x,z)= (-\tau (x,z))^{\mathbbm{1}_{z}(z') \mathbbm{1}_{x}(x')}\left| A^{x}(z'|z)Q^{z}(x'|x)\right|\mbox{.}
\end{equation} 
Observe  that (\ref{2_1}) has been defined so as  to ensure $R(x',z'|x,z)$ is non-positive only when $(x,z)=(x',z')$.
In what follows we  choose $\tau(x,z)= \tau_{\epsilon} (x,z)$ and entries $R_{\epsilon} (x',z'|c)$, $R_{\epsilon} (c|x,z)\geq 0$, and $R_{\epsilon} (c|c)\leq 0$, such that the following relations (A), (B) hold true. 
\begin{enumerate}[(A)]{\label{items_2_1}}
\item $R_{\epsilon}$ is dissipative, and $\forall \; p \in \cY$, $\sum\limits_{s \in \cY} R_{\epsilon}(s|p)=0$ (we henceforth write $R_{\epsilon}(x',z'|x,z)$, to stress the importance of the value of $\epsilon$ ). 
\item 
The function $g$ from (\ref{eq:invmesg}) yields a stationary measure for the Markov chain generated by $R_{\epsilon}$ , i.e. \newline $\sum\limits_{p \in \cY} R_{\epsilon}(s|p)g(p)=0$.
\end{enumerate}

The above would ensure (by Hille -Yosida theorems) that  $R_{\epsilon}$ generates a Markov Process on $\cY$, which has a stationary measure $g$.
The ingredients of the construction are as follows:

\begin{enumerate}
\item Given $(x_i,z_j) \in X \times Z$, choose $\tau_{\epsilon}(x_i,z_j) >0$ large enough such that 
\begin{equation}{\label{choice1}}
\sum\limits_{(x,z) \in X \times Z} R_{\epsilon}(x_i,z_j|x,z) g(x,z) \le 0
\end{equation}
and
\begin{equation}{\label{choice2}}
\sum\limits_{(x',z') \in X \times Z} R_{\epsilon}(x',z'|x_i,z_j) \le 0.
\end{equation}

\paragraph*{} Such a choice of $\tau_{\epsilon}(x_i,z_j)$ is possible because $A^{x_i}(z_j|z_j) <0$, $Q^{z_j}(x_i|x_i) <0$ and $g(x_i,z_j) > 0$, $\forall i,j$.

\item Next, we can choose $R_{\epsilon} (x_i,z_j|c) \ge 0$, and $R_{\epsilon} (c|x_i,z_j) \ge 0$, such that 
\begin{eqnarray}{\label{3_1}}
\sum\limits_{(x,z) \in X \times Z} R_{\epsilon}(x_i,z_j|x,z) g(x,z)   + R_{\epsilon} (x_i,z_j|c) \epsilon = 0 
\end{eqnarray}
and
\begin{equation}{\label{3_2}}
\sum\limits_{(x',z') \in X \times Z} R_{\epsilon}(x',z'|x_i,z_j) + R_{\epsilon} (c|x_i,z_j) = 0.
\end{equation}
\item After that, we choose $R_{\epsilon} (c|c) \le 0$, such that
 \begin{equation}{\label{3_3}}
\sum\limits_{(x,z) \in X \times Z} R_{\epsilon} (c|x,z) g(x,z) + R_{\epsilon} (c|c) \epsilon = 0 \mbox{.}
\end{equation}
\item One may show 
 \begin{equation}{\label{3_4}}
\sum\limits_{(x',z') \in X \times Z} R_{\epsilon} (x',z'|c)  + R_{\epsilon} (c|c)  = 0 \mbox{.}
\end{equation}
\end{enumerate}
 As a result we obtain the following
\begin{thm}{\label{thm_3_1}}
Let $Z^{x}(t) \in Z,\;X^{z}(t) \in X$ be two families of discrete time Markov chains indexed by $x \in X$ and $z \in Z$ respectively, with generators $A^x,\; Q^z$ and stationary measures $\nu ^x(z)$ and $m^z(x)$ respectively. 
Given $(x_i,z_j),i = 1,...N,j = 1,...,M, \epsilon > 0$ there exist $\tau_\epsilon (x_i,z_j)$ and an additional state $c$ such that (5)-(10) hold. 
Hence, conditions (A), (B) are satisfied, $R_\epsilon$ is a generator of a discrete time Markov chain $Y(t) \in \cY$ and the function $g$ satisfying
$g(y) = m^z(x) \nu^x(z)$ for $y = (x,z)$ and $g(c) = \epsilon$ stands for its stationary measure.

\end{thm}

$Y_t$ is called the combined Markov chain.

\begin{remarks}
\begin{enumerate}
\item In the definition of $R_{\epsilon} (x',z'|x,z)$ (see (\ref{2_1})) one may use $A^{x'}(z'|z)$ instead of $ A^{x}(z'|z)$ and/or $Q^{z'}(x'|x)$ instead of $Q^{z}(x'|x)$. The  above arguments would still hold true.

\item Here is a intuitive way to make sense of transition state $c$. Suppose that instead of having to satisfy (A), (B)  we are interested in satisfying (A) only. We could then choose $\tau_{\epsilon} (x_i,z_j)>0, \; \forall i,j$ so that 
\begin{align}\label{remark_intro_eq_1}
\sum\limits_{(x,z) \in X \times Z} R_{\epsilon}(x',z'|x_i,z_j) =0.
\end{align}
Moreover,  we do not have to work with (\ref{choice1}) and (\ref{choice2}), and add state $c$. Using (\ref{remark_intro_eq_1}) only  we can define a dissipative matrix $R'_{\epsilon}$ on $X \times Z$. Let $Y'_t$ be a Markov chain generated by $R'_{\epsilon}$, with a stationary measure $g'$. When we run chain $Y'_t$, it wanders into states that $g$ does not favor (that is, states to which $g$ assigns a small mass  compared to $g'$). At this point process $Y_t$ would have an option of switching to $c$ and emerging in a region more favorable for $g$.

\item The fact that $\epsilon >0$ can be chosen arbitrarily small comes from the fact $c$ can  have  a jump rate arbitrarily large i.e transitions through $c$ can be made arbitrarily fast. It means that regardless of how large $g(X \times Z)$ is, one can always chose $g(c)= \epsilon$. The way the combined process $Y_t$ uses $c$ is a significant topic in this article.

\item A somewhat surprising feature of our construction is that it is based on the product of entries of marginal generators, rather than their sums as was the case in \cite{belopolskaya2015models}. This represents both a novelty and a challenge for future works as it re-ignites an old question of how to produce a new Markov chain from previously given ones in a meaningful manner. From this point view, an extension of the current construction to more general classes of Markov processes, particularly, diffusions would be interesting.
\end{enumerate}
\end{remarks}
\end{subsection}
\begin{subsection}{A minimal version of the single transitional state construction}
In (\ref{choice1}) and (\ref{choice2}), we set a criterion for choosing $\tau_{\epsilon}(x_i,z_j) >0$, admitting a wide range of legitimate choices (we simply say \enquote{ choose $\tau_{\epsilon}(x_i,z_j) >0$, large enough}). From this point on we shall narrow this down by imposing a stronger condition upon $\tau_{\epsilon}(x_i,z_j)$. This yields a unique $\tau_{\epsilon}(x_i,z_j)$, providing a minimal choice among possibilities allowed  by (\ref{choice1}) and (\ref{choice2}). To this end, we define
\begin{equation}\label{minimal_choice}
\begin{array}{l}
\tau_{\epsilon}(x_i,z_j)=\mbox{max}\bigg\{  \sum\limits_{\substack{(x,z) \in X \times Z \\ (x,z) \neq (x_i,z_j)}} 
\diy\frac{ \left|A^{x}(z_j|z)Q^{z}(x_i|x)\right|g(x,z)\;\;}{A^{x_i}(z_j|z_j)Q^{z_j}(x_i|x_i)g(x_i,z_j)}, \\
\qquad \qquad \qquad \qquad\sum\limits_{\substack{(x',z') \in X \times Z \\ (x',z') \neq (x_i,z_j)}} \diy\frac{ \left|A^{x_i}(z'|z_j)Q^{z_j}(x'|x_i)\right|}{A^{x_i}(z_j|z_j)Q^{z_j}(x_i|x_i)}\bigg\}.
\end{array}\end{equation}
Note that $\tau_{\epsilon}(x_i,z_j)$ from \eqref{minimal_choice} satisfies (\ref{choice1}), (\ref{choice2}). It depends on $g$, which is the product measure we are targeting. Moreover, for $(x,z) \in X \times Z$, either the rate $R_\epsilon(c|x,z)=0$ or $R_\epsilon (x,z|c)=0$. Essentially the choice of $\tau_{\epsilon}$ in \eqref{minimal_choice} prevents the combined process from going in cycles around $c$. It 
partitions the natural space $X \times Z $ into three types of states. The first type are those states 
from which the combined process can jump to $c$, the second are states where jumps from $c$ are possible. The third (possibly empty) set of states are those which allow no direct jumps to or from $c$.

\begin{remark}
An alternative way of choosing $\tau_{\epsilon} >0$ such that properties (\ref{choice1}) and (\ref{choice2}) hold true is to put $\tau_{\epsilon}(x_i,z_j) = \tau$ $\forall\; i,j$, where the value $\tau$ is defined as the smallest positive number such that \ 
$\forall \;x_i, \;z_j$,
\begin{align*}
&\sum\limits_{(x,z) \in X \times Z} (-\tau)^{\mathbbm{1}_{z_j}(z) \mathbbm{1}_{x_i}(x)} \left|A^{x}(z_j|z)Q^{z}(x_i|x)\right|g(x,z) \le 0,\;\;\hbox{and}\\
&\sum\limits_{(x',z') \in X \times Z} (-\tau)^{\mathbbm{1}_{z_j}(z') \mathbbm{1}_{x_i}(x')} \left| A^{x_i}(z'|z_j)Q^{z_j}(x'|x_i)\right| \le 0.
\end{align*}
Such a choice is more faithful to the original families $A^{x},Q^{z}$, as it changes all speeds uniformly. So we can call it a uniform choice. In this paper the uniform choice 
is not used, for it does not partition $X\times Z$ into the three types described above.
\end{remark}
\end{subsection}

\end{section}

\begin{section}{Application of the minimal construction to queueing theory}\label{example}
We apply the basic construction to an example. The example captures a phenomenon in the queueing systems, 
exhibits a general interaction between the environment and basic processes where the use of the transition 
state $c$ is transparent. 

\begin{subsection}{A formal description}\label{Formal_description_of_the_example}
Here we set $X= \{0,1,2, \cdots N\}$ and $Z=\{\mu_0, \mu_1, \cdots \mu_{M}\}$, where $\mu_{0}=\frac{1}{2} \mbox{,}\; \mu_{j}=\frac{1}{2} + \frac{j}{M}, j=1,2,\dots M $. Intuitively, a  basic state represents the number of jobs in the queue whereas the state of the environment represents the efficiency of the server. A higher value of $\mu$ leads to a higher output rate of the server. For a fixed $\mu \in Z$,
\begin{eqnarray}
\begin{split}
Q^{\mu}(i+l|i)&= 0 \mbox{, \;} \forall\; l \mbox{\; such that }\; |l|>1, \\
Q^{\mu}(i+1|i)&= 1 \mbox{, \; when \;} i = 0,1,\dots N-1, \\
Q^{\mu}(i-1|i)&= \mu \mbox{, \; when \;} i =1,2, \dots N, \\
Q^{\mu}(i|i)&=-(Q^{\mu}(i-1|i)+Q^{\mu}(i+1|i))=- (\mu+1),\; i=1,\dots N-1,\\
  Q^{\mu}(0|0)&= -1, \;\; Q^{\mu}(N|N)= -\mu .
\end{split}
\end{eqnarray}
Observe that $m^{\mu}(i)=\diy\frac{1}{\eta(\mu)}\frac{1}{\mu^{i}}$ is a stationary 
probability measure for the process, where $\eta(\mu)= \sum\limits_{i=0}^{N} \frac{1}{\mu^{i}}$ is a normalizing constant. (It is the only stationary probability measure, as the basic process is irreducible). Next for a fixed $i \in X \setminus \{0\}$,
\begin{eqnarray}
\begin{split}
A^{i}(\mu_{j+l}|\mu_{j})&= 0, \;\forall\; l \mbox{\; such that }\; |l|>1, \\
A^{i}(\mu_{j+1}|\mu_{j})&= i \mbox{, \; when \;} j=0,\dots M-1, \\
A^{i}(\mu_{j-1}|\mu_{j})&= 1 \mbox{, \; when \;} j=1, \dots M, \\
A^{i}(\mu_{j}|\mu_{j})&=-(A^{i}(\mu_{j+1}|\mu_{j})+ A^{i}(\mu_{j-1}|\mu_{j}))= -(i+1),\;
j=1,\dots M-1,\\
A^{i}(\mu_{0}|\mu_{0})&=-i, \;\; A^{i}(\mu_{0}|\mu_{0})=-1.\;\\
\end{split}
\end{eqnarray}
The Markov chain on $Z$ with generator $A^{i}$ has a unique stationary probability measure $\nu^{i}(\mu_j)= \frac{i^{j}}{\sigma(i)}$ where $\sigma(i)=\sum\limits_{j=0}^{M}i^{j}$. For $i=0$ we set 
\begin{align}
\begin{split}
A^{0}(\mu_{j+l}|\mu_{j})&=0,\; \forall\; l \mbox{\; such that }\; |l|>1, \\
A^{0}(\mu_{j+1}|\mu_{j})&= \frac{1}{M^2},\;\;j=0,\dots, M-1,\\
A^{0}(\mu_{j-1}|\mu_{j})&= 1,\;\; j=1, \dots, M,\\
A^{0}(\mu_{j}|\mu_{j})&= -1- \frac{1}{M^2}, j=1, \dots M-1,\\
A^{0}(\mu_{0}|\mu_{0})&= - \frac{1}{M^2}, \; A^{0}(\mu_{M}|\mu_{M})= - 1.
\end{split}
\end{align}
The stationary measure for generator $A^{0}$ is $\nu^{0}(\mu_j)=  \frac{1}{M^{2j}\sigma(0)} $ where  $\sigma(0)= 1+ o(\frac{1}{M^2})$.
As before, we construct a dissipative operator $R_{\epsilon}$ on $\cY= \{X \times Z\} \cup \{c\}$ by using \eqref{2_1}:
$$R_{\epsilon} (i',\mu'|i,\mu)= (-\tau_{\epsilon} (i,\mu))^{\mathbbm{1}_{\mu}(\mu') \mathbbm{1}_{i}(i')}| A^{i}(\mu'|\mu)Q^{\mu'}(i'|i)| \;\text{for} \; (i,\mu),(i',\mu') \in X \times Z$$. 
Again choose $\tau_{\epsilon}(i,\mu)$ minimally as in (\ref{minimal_choice}). Next, set: 
\begin{align}
\begin{split}
g((i,\mu))&=m^{\mu}(i)\nu^{i}(\mu) \mbox{\;,\;} \forall (i,\mu) \in X \times Z  \\
g(c) &=\epsilon \mbox{.}
\end{split}
\end{align}
Further define $R_{\epsilon} (i',\mu'|c)$, $R_{\epsilon} (c|i,\mu)$ and $R_{\epsilon} (c|c)$ by using (\ref{3_1}), (\ref{3_2}) and (\ref{3_3}). From Theorem \ref{thm_3_1} conclude that $R_{\epsilon}$ is a dissipative operator, it is the generator of a Markov process on $\cY$, with $g$ serving as a unique stationary measure.
As $M$  increases, we can see when the process jumps from the natural space $X \times Z$ to $c$. Indeed, $M \to \infty$ means that the spacing between consecutive levels of the server output is shrinking to zero. (Formally, the process on $Z$, for a fixed $i$ becomes like a diffusion.)
\begin{remark}
An intuitive picture of how the basic and environment processes interact is as follows:
\begin{enumerate}
\item For a fixed  $\mu \in Z$, the basic process is a continuous-time Markov chain 
on $X$. The state of the basic process captures the number of jobs in the queue. The arrival rate of jobs in the queue is $1$, the rate at which jobs are cleared from the queue is $\mu$. We think of $\mu$ as the efficiency at which the server is operating.
\item For fixed  $i>1$, the server looks to increase the output performance captured by the value of $\mu$. This is achieved by a drift towards higher values of $\mu$, and the  drift grows with  $i$. The value of $i$ captures the number of jobs in the queue: the more the number of jobs the more fervently the server tries to drive towards higher efficiency . When $i=0$, the server  tries to save power by developing a drift towards lower productivity. 
\item State $c$ can be interpreted as a maintenance state which the process attains to do  repairs. Exactly when a repair is needed is evident from specifying the states from which jumps to $c$ are possible. The nature of performed repairs will be revealed by identifying the states to which jumps from $c$ are plausible.
\item In its present form the construction works only when the number of basic and environment states are both finite.
That is, we can't  directly put $M= \infty$ or $N = \infty$ because the algorithm of finding the elements of the matrix $R_{\epsilon}$ must terminate. This makes it interesting to analyze the limits as $M$, $N$ $\to \infty$.
\end{enumerate}

\end{remark}
\end{subsection}

\begin{subsection}{A technical analysis}
Throughout the rest of the article $N$ is considered fixed. And we will study what happens as $M \to$  infinity. Remark 3 in Section 03, makes it important to delineate from which states in the natural space jumps to $c$ are possible, and which states jumps from $c$ are possible. The main tool for studying this will be Theorem \ref{delienation theorem}. The final conclusions are given in the remarks at the end of the subsection,  before doing all this we make some technical considerations.\\
 Observe that $ \forall \; i , \;j$, Next:
$\nu^{i}(\mu_{j+1})=i\nu^{i}(\mu_{j})$  and $m^{\mu_{j+1}}(i)= m^{\mu_{j}}(i) (1+o(\frac{1}{M^{\frac{7}{8}}}))$.
\begin{align}\label{recurrence_2}
\begin{split}
\frac{\nu^{i+1}(\mu_j)}{\nu^{i}(\mu_j)} &= \left. \frac{(i+1)^{j}\;i}{(i+1)^{M+1}-1} \right/ \frac{i^{j}(i-1)}{(i)^{M+1}-1}\\
&= \frac{i^{M+2-j}}{(i+1)^{M+1-j}(i-1)}(1+o(\frac{1}{M})) \mbox{.}\\
\frac{\nu^{i-1}(\mu_j)}{\nu^{i}(\mu_j)} &= 
\frac{i^{M+1-j}}{(i-1)^{M+2-j}}(i-2)(1+o(\frac{1}{M})),\; \forall \; i>2 \mbox{.} \\
\end{split}
\end{align}

The foregoing argument implies the following recurrence relations: when $i \neq {1,0}$,
\begin{align}\label{recur1}
\begin{split}
g(i,\mu_{j+1})&=i g(i,\mu_{j})(1+o(\frac{1}{M^{\frac{7}{8}}})), \text{hence}\\
g(i,\mu_{j+l})&=i^l g(i,\mu_{j})(1+o(\frac{1}{M^{\frac{1}{9}}})), \text {when}\;  0<l< M^{\frac{3}{4}} \mbox{;}\\
e^{-C M^{\frac{1}{8}}}i^l g(i,\mu_{j})&(1+o(\frac{1}{M^{\frac{1}{9}}}))\le g(i,\mu_{j+l})\le e^{C M^{\frac{1}{8}}}i^l g(i,\mu_{j})(1+o(\frac{1}{M^{\frac{1}{9}}})).
\end{split}
\end{align}
Here\; $C$ is a positive contant, depending on $N$. Furhtermore
\begin{align}\label{recur2}
\begin{split}
g(i,\mu_{j-1})&=\frac{1}{i} g(i,\mu_{j})(1+o(\frac{1}{M^{\frac{7}{8}}})),\\
g(i+1,\mu_{j})&=g(i,\mu_{j}) \frac{1}{\mu_{j}}\frac{i^{M+2-j}}{(i+1)^{M+1-j}(i-1)}(1+o(\frac{1}{M})) \mbox{\;(when \;} i \ge 2),\\
g(i-1,\mu_{j})&= g(i,\mu_{j})\mu_{j} \frac{i^{M+1-j}(i-2)}{(i-1)^{M+1-j}(i-1)}(1+o(\frac{1}{M})) \mbox{\;(when \;} i \ge 3) \mbox{.}
\end{split}
\end{align}
Observe that the term denoted as $o(\frac{1}{M})$ above  can be chosen so that its 
decay as $M \to\infty$  depends only on $N$ and is uniform in $i$ and $j$. (This fact will be used in subsequent calculations.) Formally this would be the statement that given $\varepsilon>0$, there exists $M''>0$, such that $\forall \; M'> M''$ and $\forall \; i,\;j$, we have $M'^{\frac{7}{8}} [\frac{g(i,\mu_{j+1})}{ig(i,\mu_{j})}-1] \le \varepsilon$. A similar uniform bound holds for (\ref{recurrence_2}) (\ref{recur1}) and (\ref{recur2}).\\
The following theorem follows trivially from the construction
\begin{thm}{\label{delienation theorem}}
A jump from $(i, \mu_{j})$  to $c$ is possible if and only if
\begin{align}\label{nec_cond}
\begin{split}
&\sum\limits_{\substack{(i', \mu_{j'}) \in X \times Z \\ (i', \mu_{j'}) \neq (i, \mu_{j})}} \frac{ |A^{i'}(\mu_{j}|\mu_{j'})Q^{\mu_{j}}(i|i')|g((i', \mu_{j'}))} {A^{i}(\mu_{j}|\mu_{j})Q^{\mu_{j}}(i|i)g(i,\mu_{j})} 
 \ge  \sum\limits_{\substack{(i', \mu_{j'}) \in X \times Z \\ (i', \mu_{j'}) \neq (i, \mu_{j})}}  \frac{|A^{i}(\mu_{j'}|\mu_{j})Q^{\mu_{j'}}(i'|i)|}  {A^{i}(\mu_{j}|\mu_{j})Q^{\mu_{j}}(i|i)} \\
 & \Leftrightarrow \sum\limits_{\substack{(i', \mu_{j'}) \in X \times Z \\ (i', j') \neq (i, j)}} | A^{i'}(\mu_{j}|\mu_{j'})Q^{\mu_{j}}(i|i')|g((i', \mu_{j'})) 
\ge \sum\limits_{\substack{(i', \mu_{j'}) \in X \times Z \\ (i', \mu_{j'}) \neq (i, \mu_{j})}} |A^{i}(\mu_{j'}|\mu_{j})Q^{\mu_{j'}}(i'|i)|g(i,\mu_{j}) \\
 &\Leftrightarrow \tau_{\epsilon} (i,\mu_{j})= \sum\limits_{\substack{(i', \mu_{j'}) \in X \times Z \\ (i', \mu_{j'}) \neq (i, \mu_{j})}} \frac{ | A^{i'}(\mu_{j}|\mu_{j'})Q^{\mu_{j}}(i|i')| \;g((i', \mu_{j'}))}  {A^{i}(\mu_{j}|\mu_{j})Q^{\mu_{j}}(i|i)g(i,\mu_{j})} \mbox{.}
\end{split}
\end{align}
\end{thm}
In what follows we characterize when (\ref{nec_cond}) holds true for $M$ large enough. We will  use the abbreviation $g(i,\mu_{j})= g(i,j)$.
When $i \neq 0,1,N$ and $j \neq 0,M$, we have that
\begin{align}\label{row_sums}
\begin{split}
&\sum\limits_{\substack{(i', \mu_{j'}) \in X \times Z \\ (i', \mu_{j'}) \neq (i, \mu_{j})}} (-1)^{\mathbbm{1}_{j}(j')}(-1)^{\mathbbm{1}_{i}(i')}A^{i}(\mu_{j'}|\mu_{j})Q^{\mu_{j'}}(i'|i)g(i,j) \\= &g(i,j) \;\left[2(\mu_{j-1}+1)+(i+1)(\mu_{j}+1)+2i(\mu_{j+1}+1)\right]\\
&= g(i,j)\left[3\mu_{j}+3\mu_{j}i+3i+3 + o(\frac{1}{\sqrt{M}})\right].
\end{split}
\end{align}
The following calculation is valid for $i \neq 0,1,2$ ((\ref{recur1}) and (\ref{recur2}) are instrumental here):
\begin{align}\label{column_sums}
\begin{split}
&\sum\limits_{\substack{(i', \mu_{j'}) \in X \times Z \\ (i', \mu_{j'}) \neq (i, \mu_{j})}} \left| A^{i'}(\mu_{j}| \mu_{j'})Q^{\mu_{j}}(i|i')\right|g((i',j')) \\
&\qquad\qquad = \mu_{j}(g(i+1,j-1) +2g(i+1,j)+g(i,j+1)+g(i+1,j+1))\\
&\qquad\qquad\qquad\qquad+ \mu_{j}i(g(i,j-1)+g(i+1,j-1)+g(i+1,j))\\
&\qquad\qquad\qquad\qquad + i(g(i-1,j-1)+g(i,j-1)+g(i-1,j))\\
&\qquad\qquad\qquad\qquad+ 1(g(i,j+1)+ g(i-1,j+1)+ g(i-1,j)- g(i-1,j-1))\\
&\qquad\qquad\qquad\qquad\ge (i-1)g(i-1,j) 
= g(i,j)\left[(\frac{i}{i-1})^{M+1-j}\mu_{j}(i-2) + o(\frac{1}{\sqrt{M}})\right].
\end{split}
\end{align}
 Now recall: $N$ is fixed and $\mu_j \le \frac{3}{2}$. As $i \neq 0,1,2$, comparing (\ref{column_sums}) and (\ref{row_sums}) shows that for  given $\epsilon' \in  (0,1)$ we can choose M large enough such that $\forall \; j \le \epsilon' \times M$ and $\forall \; i > 1$ the conditions in (\ref{nec_cond}) are satisfied. (In fact, direct computations show the above fact to be true for $i=2$.)
\begin{remark}
When $i \neq 0,1$, the environment has a propensity to move towards a higher efficiency. When $j \le \epsilon' \times M$, the system achieves this by jumping to the transition state $c$. The Minimality property in (\ref{minimal_choice}) ensures that the process  cannot revert from $c$ to a low-efficiency state. Instead, it must progress to a 
higher-efficiency of the server . It can also reject some jobs and go to a state with $i  \in  \{0,1\}$; in subsequent Sections we will refine the construction to prevent such a possibility. 
\end{remark}

To carry out calculations for $i=0$, observe that for $j \ne 0$:
\begin{equation}
\sum\limits_{\substack{(i', \mu_{j'}) \in X \times Z \\ (i', \mu_{j'}) \neq (0, \mu_{j})}} |A^{0}(\mu_{j'}|\mu_{j})Q^{\mu_{j'}}(i'|0)|g(0,j)= o(\frac{1}{M^2}) \times \frac{1}{\eta(\mu_j)} ,\end{equation}
and
\begin{equation}\begin{array}{l}
\sum\limits_{\substack{(i', \mu_{j'}) \in X \times Z \\ (i', \mu_{j'}) \neq (0, \mu_{j})}}  |A^{i'}(\mu_{j}| \mu_{j'})Q^{\mu_{j}}(0|i')|g((i',j'))\\
\qquad\qquad\qquad \ge A^{1}(\mu_{j}| \mu_{j-1})Q^{\mu_{j}}(0|1)g((1,j-1))
= \frac{\mu_{j}}{\mu_{j-1}}\frac{1}{M} \frac{1}{\eta(\mu_{j-1})}\mbox{.}
\end{array}
\end{equation}
Now, for $M$ large enough, by virtue of (\ref{nec_cond}) there is a possibility of jump from $(0, \mu_{j})$ to $c$, when $j \ne 0$. On the  other hand, 
when $j=0$, a comparison of the two terms in (\ref{nec_cond}) reveals that the process can jump from $c$ to $(0,\mu_{0})$. Then minimality guarantees that a jump in the opposite direction is impossible.

\begin{remark}
When $i=0$, the natural disposition for the process is to attain the lowest efficiency. If  $\mu \ne \frac{1}{2}$, the process jumps to $c$ from a state where by minimality (\ref{minimal_choice}) it may go to $\mu= \frac{1}{2}$. Or it can create some jobs for itself and attain a state $i> 0$. This is impractical but possible under the present construction. In the next Section we will make revisions to outlaw this.
\end{remark}
\end{subsection}
\end{section}

\begin{section}{A construction with multiple maintenance states}\label{construc_multiple_maintainance_states}
The remarks in Section 3 necessitate a construction more suitable to the queuing system
philosophy. Specifically the maintenance state in the single state construction for the queuing example may create or delete jobs which is undesirable. To this end we introduce here multiple transition states. We first describe the construction in the the specific case of the queuing example, section 4.2 has theorems which evince that indeed in this multiple maintenance state framework the above mentioned creation and deletion of jobs has \enquote{negligible} probability. Section 4 shows that not only the construction is robust enough to handle generalizations to multiple transition states, but also such generalizations are often intuitively necessary. \par Suppose sets $X$, $Z$,  families $A^{i},\; Q^{\mu} $ and functions $m^{\mu}(i), \; \nu^{i}(\mu)$ are as in Section 3. Set $\widehat{\cY}= X \times Z \cup \{c_0,c_1,c_2, \cdots c_N\}$. We will construct a dissipative matrix $R_{\epsilon}$ on $\widehat{\cY}$ by using the family of operators $A^{i}$, and $Q^{\mu}$, which will have $g$ (defined below) as a stationary measure. We set
\begin{align}\label{stat_measure_for_multiple_main_states}
\begin{split}
g((i,\mu))&=m^{\mu}(i)\nu^{i}(\mu) \mbox{\;,\;} \forall (i,\mu) \in X \times Z  \\
g(c) &=\epsilon, \forall \; c  \in \{c_0,c_1,c_2, \cdots c_N\} \mbox{.}
\end{split}
\end{align}
The construction below is similar to the one described in Section \ref{general_construction}, except for multiplicity of transient states. This can be understood as follows, in Section \ref{Formal_description_of_the_example} we partition the natural space $X \times Z$ into $N+1$ parts according to the number of jobs in the queue. We now make sure $c_i$ interacts only with elements of the $i$-th partition. There might still be jumps between the $c_i$'s, the probability of such jumps is shown to be very small (see Theorems \ref{first_theorem_last_section} and \ref{second_theorem_last_section} ). Thus the maintenance states only alter the server efficiency, rarely creating or rejecting jobs. The construction can be applied in a variety of settings where there is a partition of the natural space as described above.

\begin{subsection}{Description of the refined construction on the example}\label{Construction_of_combined process}
We construct $R_{\epsilon}$, a matrix of dimension $|\widehat{\cY}| \times |\widehat{\cY}|$. We describe its entries in a sequence of steps. In order to ensure the matrix is  dissipative, the row sums must be zero, and all non-diagonal entries non-negative. We  use the notation introduced in Section \ref{Introduction}. 
\begin{enumerate}

\item  Given $ (i,\mu),(i',\mu') \in X \times Z.$Define
\begin{equation}
 R_{\epsilon} (i',\mu'|i,\mu)= (-\tau_{\epsilon} (i,\mu))^{\mathbbm{1}_{\mu}(\mu') \mathbbm{1}_{i}(i')} \left| A^{i}(\mu'|\mu)Q^{\mu'}(i'|i)\right|
\end{equation}
where
\begin{equation}\label{tau_multiple_state_construction}
\tau_{\epsilon}(i,\mu)=
\mbox{max}\left\{  \sum\limits_{\substack{(i',\mu') \in X \times Z \\ (i',\mu') \neq (i,\mu)}} \frac{ |A^{i'}(\mu|\mu')Q^{\mu}(i|i')|g(i',\mu')} {A^{i}(\mu|\mu)Q^{\mu}(i|i)g(i,\mu)} ,
 \sum\limits_{\substack{(i',\mu') \in X \times Z \\ (i',\mu') \neq (i,\mu)}} \frac{ |A^{i}(\mu'|\mu)Q^{\mu'}(i'|i)|}{A^{i}(\mu|\mu)Q^{\mu}(i|i)} \;\;\right\}.
\end{equation}

\item Next, set $\forall \; (i,\mu) \in X \times Z, \; j \in X, \; j \neq i \;  R_{\epsilon}(c_j|(i,\mu))=R_{\epsilon}((i,\mu)|c_j)=0$. 
Furthermore $\forall \; (i,\mu) \in X \times Z$ define
\begin{equation*}
R_{\epsilon}(c_i|(i,\mu)) = \begin{cases}
&0, \; \; \text{when  $\; \tau_{\epsilon}(i,\mu)= \sum\limits_{\substack{(i',\mu') \in X \times Z \\ (i',\mu') \neq (i,\mu)}}\;
\diy\frac{\left| A^{i}(\mu'|\mu)Q^{\mu'}(i'|i)\right|\;\;}{A^{i}(\mu|\mu)Q^{\mu}(i|i)}$}\\
&-\sum\limits_{(i',\mu') \in X \times Z}R_{\epsilon}((i',\mu')|(i,\mu)), \; \;\text{otherwise}.
\end{cases}
\end{equation*}

Note that the way in which we constructed $\tau_{\epsilon}(i,\mu)$ in (\ref{1_1}) ensures that $R_{\epsilon}(c_i|(i,\mu)) \ge 0$. 
\item Define 
\begin{equation*}
R_{\epsilon}((i,\mu)|c_i) = \begin{cases}
&0, \; \; \text{when  $\; \tau_{\epsilon}(i,\mu)= \sum\limits_{\substack{(i',\mu') \in X \times Z \\ (i',\mu') \neq (i,\mu)}}\; 
\diy\frac{ \left| A^{i'}(\mu|\mu')Q^{\mu}(i|i')\right|g((i',\mu'))}{A^{i}(\mu|\mu)Q^{\mu}(i|i)g(i,\mu)}$}\\
&-\sum\limits_{(i',\mu') \in X \times Z}\diy\frac{R_{\epsilon}((i,\mu)|(i',\mu'))g(i',\mu')}{\epsilon},  \; \;\text{otherwise}.
\end{cases}
\end{equation*}
Again, note that $R_{\epsilon}((i,\mu)|c_i) \ge 0$. The choices made thus far guarantee that 
$\forall (i,\mu) \in X \times Z$, we have
\begin{align}\label{init_crit_for_multiple_states}
&\sum\limits_{(i',\mu') \in X \times Z}R_{\epsilon}((i',\mu')|(i,\mu)) + R_{\epsilon}(c_i|(i,\mu))=0 \nonumber\\
& \sum\limits_{(i',\mu') \in X \times Z}R_{\epsilon}((i,\mu)|(i',\mu'))g(i',\mu') +
R_{\epsilon}((i,\mu)|c_i)\times \epsilon =0.
\end{align}
\item All that remains is to choose $R_{\epsilon}(c_j|c_i)$, where $i,j = \{0,1 \ldots, N\}$. For any $i,j$ with $|i-j|>1$, set $R_{\epsilon}(c_j|c_i)=0$
\item For $|i-j| \le 1$ we define $R_{\epsilon}(c_j|c_i)$ in a recurrent manner.
First, set

\begin{align}\label{definition_V's}
\begin{split}
{V_r}^{(N)}&= \sum\limits_{\mu \in Z}R_{\epsilon}(c_N|(N,\mu))\times g(N,\mu)\mbox{,}\\
{V_c}^{(N)}&= \sum\limits_{\mu \in Z}R_{\epsilon}((N,\mu)|c_N)\times \epsilon \mbox{,}\\
{V_m}^{(N)}&= \text{max}\{{V_r}^{(N)}, {V_c}^{(N)}\}\mbox{.}
\end{split}
\end{align}

Next, choose $R_{\epsilon}(c_N|c_N)= \diy -\frac{{V_m}^{(N)}}{\epsilon}$ and
\begin{equation*}
R_{\epsilon}(c_{N-1}|c_N) = \begin{cases}
&0, \; \text{when ${V_c}^{(N)} \ge {V_r}^{(N)}$}\\
& \diy \frac{{V_r}^{(N)}-{V_c}^{(N)}}{\epsilon}, \; \text{otherwise}.
\end{cases}
\end{equation*}
Also set
\begin{equation*}
R_{\epsilon}(c_N|c_{N-1}) = \begin{cases}
&0, \; \text{when ${V_c}^{(N)} \le {V_r}^{(N)}$}\\
& \diy \frac{{V_c}^{(N)}-{V_r}^{(N)}}{\epsilon}, \; \text{otherwise}.
\end{cases}
\end{equation*}
These choices  ensure that,
\begin{align}\label{eqns_for_c_N}
\begin{split}
&R_{\epsilon}(c_{N-1}|c_N) \ge 0, \; R_{\epsilon}(c_N|c_{N-1}) \ge 0,\;  R_{\epsilon}(c_N|c_N) <0, \; R_{\epsilon}(c_{N-1}|c_N)R_{\epsilon}(c_N|c_{N-1})=0\\
 &\sum\limits_{\mu \in Z}R_{\epsilon}((N,\mu)|c_N) + R_{\epsilon}(c_N|c_N) + R_{\epsilon}(c_{N-1}|c_N)=0 \\
& \sum\limits_{\mu \in Z}R_{\epsilon}(c_N|(N,\mu))g(N,\mu)+\epsilon\times(R_{\epsilon}(c_N|c_N)+R_{\epsilon}(c_N|c_{N-1}))=0\mbox{.}
\end{split}
\end{align}
This indicates that either jumps from $c_N$ to $c_{N-1}$ are or jumps from $c_{N-1}$ to $c_N$ are impossible.
\item We now define the quantities $R_{\epsilon}(c_l|c_{l'})$, with $|l-l'| \le 1, \; l' \ne 0$. The case where $l'=N$ or $l=N$ has by this time been dealt with.  we define the quantities $R_{\epsilon}(c_j|c_{j-1})$, $R_{\epsilon}(c_{j-1}|c_j)$ and $R_{\epsilon}(c_j|c_j)$ with $j>0$, given we  have already
defined $R_{\epsilon}(c_j|c_{j+1})$ and $R_{\epsilon}(c_{j+1}|c_j)$. By recursion we will have 
defined $R_\epsilon(c_i|c_j)$, $\forall \; i, \; j = 0,1,\ldots M$, except for $i=j= 0$. Set:
\begin{align}\label{definition_V's_genj}
\begin{split}
{V_r}^{(j)}&= \sum\limits_{\mu \in Z}R_{\epsilon}(c_j|(j,\mu))\times g(j,\mu) + \epsilon\times R_{\epsilon}(c_j|c_{j+1})\mbox{.}\\
{V_c}^{(j)}&= \sum\limits_{\mu \in Z}R_{\epsilon}((j,\mu)|c_j)\times \epsilon+ R_{\epsilon}(c_{j+1}|c_j)\times \epsilon \mbox{.}\\
{V_m}^{(j)}&= \text{max}\{{V_r}^{(j)}, {V_c}^{(j)}\} \mbox{.}
\end{split}
\end{align}

Choose $R_{\epsilon}(c_j|c_j)= -\diy \frac{{V_m}^{(j)}}{\epsilon}$. Choose
\begin{equation*}
R_{\epsilon}(c_{j-1}|c_j) = \begin{cases}
&0, \; \text{when ${V_c}^{(j)} \ge {V_r}^{(j)}$}\\
& \diy\frac{{V_r}^{(j)}-{V_c}^{(j)}}{\epsilon}, \; \text{otherwise}.
\end{cases}
\end{equation*}
Also select
\begin{equation*}
R_{\epsilon}(c_j|c_{j-1}) = \begin{cases}
&0, \; \text{when ${V_c}^{(j)} \le {V_r}^{(j)}$}\\
& \diy\frac{{V_c}^{(j)}-{V_r}^{(j)}}{\epsilon}, \; \text{otherwise}.
\end{cases}
\end{equation*}
Note that these choices  ensure 
\begin{align}\label{eqns_for_c_j}
\begin{split}
&R_{\epsilon}(c_{j-1}|c_j) \ge 0, \; R_{\epsilon}(c_j|c_{j-1}) \ge 0,\;  R_{\epsilon}(c_j|c_j) <0, \; R_{\epsilon}(c_{j-1}|c_j)R_{\epsilon}(c_j|c_{j-1})=0\mbox{,}\\
 &\sum\limits_{\mu \in Z}R_{\epsilon}((j,\mu)|c_j) + R_{\epsilon}(c_j|c_j) + R_{\epsilon}(c_{j-1}|c_j)+  R_{\epsilon}(c_{j+1}|c_j)=0 \mbox{,}\\
& \sum\limits_{\mu \in Z}R_{\epsilon}(c_j|(j,\mu))g(j,\mu)+\epsilon \bigg(R_{\epsilon}(c_j|c_j)+R_{\epsilon}(c_j|c_{j-1})+ R_{\epsilon}(c_j|c_{j+1})\bigg)=0.
\end{split}
\end{align}

\item Finally when $j=0$, we define $R_{\epsilon}(c_0|c_0)=- \left[\sum\limits_{\mu \in Z}R_{\epsilon}((0,\mu)|c_0) +  R_{\epsilon}(c_1|c_0) \right]$, this implies that $\sum\limits_{\mu \in Z}R_{\epsilon}((0,\mu)|c_0) + R_{\epsilon}(c_0|c_0) + R_{\epsilon}(c_1|c_0)=0$. One can prove an assertion similar to Theorem \ref{thm_3_1} to get that 
$$
 \sum\limits_{\mu \in Z}R_{\epsilon}(c_0|(0,\mu))g(0,\mu)+\epsilon\bigg(R_{\epsilon}(c_0|c_0)+ R_{\epsilon}(c_0|c_1)\bigg)=0.
$$
\end{enumerate}
Thus we have now defined a combined process on $\widehat{\cY}$ with $g$ (as defined in 
(\ref{stat_measure_for_multiple_main_states})) as a stationary measure.
\end{subsection}

\begin{subsection}{Analysis of transitional rates}
In the above construction, each state $c_i$ can be connected to $(j,\mu)$ only if $j=i$. Consequently, the only way for jobs to be created or deleted during maintenance is by communications between the various $c_i$'s. However the probabilities $\frac{R(c_j|c_i)}{|R(c_i|c_i)|}$ will be shown to be \enquote{small} in this section. 
 
We start our analysis with $i=N$, and then proceed to smaller $i$'s.
\begin{thm}\label{first_theorem_last_section}
For $N \ge 2$
$$\max \Big[\;R_\epsilon(c_N|c_{N-1}), R_\epsilon(c_{N-1}|c_N)\; \Big] 
= o(\frac{1}{M^{\frac{1}{9}}})\times\; \min\Big[\;|R_\epsilon(c_N|c_{N})|,\;
|R_\epsilon(c_{N-1}|c_{N-1})|\;\Big].$$ 
\end{thm}
\begin{proof}
 For simplicity we give the proof for  $N> 2$ (See Remark at the end).\newline
We attempt to estimate rates $R_{\epsilon}(c_{N-1}|c_N)$ and $R_{\epsilon}(c_N|c_{N-1})$. Note that 
\begin{equation*}
\max \Big[R_\epsilon(c_N|c_{N-1}), R_\epsilon(c_{N-1}|c_N)\; \Big]= 
\frac{|{V_c}^{(N)}-{V_r}^{(N)}|}{\epsilon},\;\;\;
\min \Big[R_\epsilon(c_N|c_{N-1}), R_\epsilon(c_{N-1}|c_N)\; \Big]= 0.
\end{equation*}
So $|{V_c}^{(N)}-{V_r}^{(N)}|$ (Cf \eqref{definition_V's}) is the main term that needs to be estimated. To this end, set 
\begin{align}{\label{def_tau_N}}
{\tau_{1}}^{(N)}= \sum\limits_{\substack{\mu_n,\mu_j, i \\ i \neq N}} |A^{N}(\mu_j|\mu_n)|Q^{\mu_j}(i|N)g(N,\mu_n);\;\;
{\tau_{2}}^{(N)}= \sum\limits_{\substack{\mu_n,\mu_j, i \\ i \neq N}} |A^{i}(\mu_j|\mu_n)|Q^{\mu_j}(N|i)g(i,\mu_n)
\end{align}
Here $g$ is as in \eqref{stat_measure_for_multiple_main_states}. The purpose of \eqref{def_tau_N} is clear from the following formulas
$$\begin{array}{cl}
{V_r}^{(N)}= -\sum\limits_{\mu_n,i,\mu_j}R_{\epsilon}((i,\mu_j)|(N,\mu_n))\times g(N,\mu_n)
\;= -\sum\limits_{\mu_n,\mu_j} R_{\epsilon}((N,\mu_j)|(N,\mu_n))g(N,\mu_n)- {\tau_{1}}^{(N)}.
\end{array}$$
\begin{eqnarray*}
{V_c}^{(N)}=-\sum\limits_{\mu_n,i,\mu_j}R_{\epsilon}((N,\mu_j)|(i,\mu_n))\times g(i,\mu_n)
=-\sum\limits_{\mu_n,\mu_j} R_{\epsilon}((N,\mu_j)|(N,\mu_n))g(N,\mu_n)- {\tau_{2}}^{(N)}.
\end{eqnarray*}
The above equations follow from Eqn (\ref{definition_V's}), dissipativity of $R_{\epsilon}$ and the fact that $g$ is 
a stationary measure for $R_{\epsilon}$.

It follows that $|{V_r}^{N}-{V_c}^{N}|= |{\tau_{1}}^{(N)}-{\tau_{2}}^{(N)}|$. Next we reduce ${\tau_{1}}^{(N)}$ and ${\tau_{2}}^{(N)}$ to more comparable expressions.

In the argument below, we use (i) dissipativity of $Q^{\mu_j}$, (ii)  $|n-j| \le 1 \; \text{implies} \; m^{\mu_n}(N)= m^{\mu_j}(N)(1+o(\frac{1}{\sqrt{M}}))$,  and (iii) the fact that $\nu^{N}$ is a stationary measure for  $A^{N}$:
\begin{align*}
{\tau_{1}}^{(N)} &=   \sum\limits_{\substack{\mu_n,\mu_j, i \\ i \neq N}
} |A^{N}(\mu_j|\mu_n)|Q^{\mu_j}(i|N)g(N,\mu_n)
=  \sum\limits_{\mu_n,\mu_j} |A^{N}(\mu_j|\mu_n)||Q^{\mu_j}(N|N)|g(N,\mu_n)\\ 
 &= \sum\limits_{\mu_j} |Q^{\mu_j}(N|N)|m^{\mu_j}(N)\sum\limits_{\mu_n}|A^{N}(\mu_j|\mu_n)|\nu^{N}(\mu_n)(1+o(\frac{1}{\sqrt{M}}))\\
 &= 2 \sum\limits_{\mu_j} |Q^{\mu_j}(N|N)|m^{\mu_j}(N)|A^{N}(\mu_j|\mu_j)|\nu^{N}(\mu_j)(1+o(\frac{1}{\sqrt{M}}))\\
 &=  2 \sum\limits_{\substack{\mu_j, i \\ i \neq N} } Q^{\mu_j}(N|i)m^{\mu_j}(i)|A^{N}(\mu_j|\mu_j)|\nu^{N}(\mu_j)(1+o(\frac{1}{\sqrt{M}})).
\end{align*}
Similarly,
\begin{align*}
{\tau_{2}}^{(N)}&= \sum\limits_{\substack{\mu_n,\mu_j, i \\ i \neq N}
} |A^{i}(\mu_j|\mu_n)|Q^{\mu_j}(N|i)g(i,\mu_n)\\
&= \sum\limits_{\substack{\mu_n,\mu_j, i \\ i \neq N}} Q^{\mu_j}(N|i)m^{\mu_j}(i)\nu^{i}(\mu_n)|A^{i}(\mu_j|\mu_n)|(1+o(\frac{1}{\sqrt{M}}))\\
&=   2 \sum\limits_{\substack{\mu_j, i \\ i \neq N} } Q^{\mu_j}(N|i)m^{\mu_j}(i)|A^{i}(\mu_j|\mu_j)|\nu^{i}(\mu_j)(1+o(\frac{1}{\sqrt{M}})).
\end{align*}

The rest of the proof will be devoted to examining $|{\tau_{1}}^{(N)}-{\tau_{2}}^{(N)}|$. Observe that
\begin{equation}\label{Interaction_c_N_C_N_ref_point_1}\begin{array}{l}
 |{\tau_{1}}^{(N)}-{\tau_{2}}^{(N)}|
 = \left|2 \sum\limits_{\substack{\mu_j, i \\ i \neq N} } Q^{\mu_j}(N|i)m^{\mu_j}(i)
\left[|A^{N}(\mu_j|\mu_j)|\nu^{N}(\mu_j)- |A^{i}(\mu_j|\mu_j)|\nu^{i}(\mu_j)\right]
(1+o(\diy\frac{1}{\sqrt{M}}))\right|
\end{array}\end{equation}
Using the fact that  the only non-zero contributions comes from $i=N-1$, the last 
expression equals
\begin{align}\label{Interaction_c_N_C_N_ref_point_1_1}
\begin{split}
&\left|\;2 \sum\limits_{\mu_j} Q^{\mu_j}(N|N-1)m^{\mu_j}(N-1)\;
(|A^{N}(\mu_j|\mu_j)|\nu^{N}(\mu_j)-\right.
\\  & \qquad \qquad|A^{N-1}(\mu_j|\mu_j) |\;\nu^{N-1}(\mu_j))\;
\left.(1+o(\frac{1}{\sqrt{M}}))\;\right|.\end{split}\end{align}
 By virtue of (\ref{recurrence_2}), (\ref{recur1}), coincides with
\begin{equation}\begin{array}{l}{\label{difference_of_taus}}
 2\Bigg| \;  \sum\limits_{j=0}^{M-1} \mu_j g(N,\mu_j)\left( N+1- (\frac{N}{N-1})^{M+2}(\frac{N-1}{N})^{j}(N-2)\right)\;\\
 \qquad\times\left(1+o(\frac{1}{\sqrt{M}})\right) 
 +\mu_M g(N,\mu_M)(1-\frac{N(N-2)}{(N-1)^2})\;\left(1+o(\frac{1}{\sqrt{M}})\right)\;\Bigg|\\
\quad =2\Bigg| \;  \sum\limits_{j=0}^{M-1} \mu_j g(N,\mu_j)( N+1- (\frac{N}{N-1})^{M+2}(\frac{N-1}{N})^{j}(N-2)  )\;\\
\qquad\times\left(1+o(\frac{1}{\sqrt{M}})\right)  
+\mu_M g(N,\mu_{M-1-M^{\frac{3}{4}}})N^{M^{\frac{3}{4}}+1}\frac{1}{(N-1)^2} \;
\left(1+o(\frac{1}{M^{\frac{1}{9}}})\right)\;\Bigg|.
\end{array}\end{equation}

The definitions below break the sum in \eqref{difference_of_taus} into two parts, which will be considered separately:
\begin{eqnarray*}
 U^{1}= \sum\limits_{j= M-1-M^{\frac{3}{4}}}^{M-1} \mu_j g(N,\mu_j)( N+1- (\frac{N}{N-1})^{M+2}(\frac{N-1}{N})^{j}(N-2)  ),\\U^{2}= \sum\limits_{j= 0}^{M-1-M^{\frac{3}{4}}} \mu_j g(N,\mu_j)( N+1- (\frac{N}{N-1})^{M+2}(\frac{N-1}{N})^{j}(N-2)  ).
\end{eqnarray*}
Formally, $ \sum\limits_{j=0}^{M-1} \mu_j g(N,\mu_j) (N+1- (\frac{N}{N-1})^{M+2}(\frac{N-1}{N})^{j}(N-2)  )= U^{1} + U^{2} $.

Using the formula for geometric progressions and the fact that 
$g(i,\mu_{j+l})=i^l g(i,\mu_{j})(1+o(\frac{1}{M^{\frac{1}{9}}}))$ when 
$0<l< M^{\frac{3}{4}}$, we obtain:
\begin{align}\label{Estimates_for_U_1}
\begin{split}
 U^{1} \;
&=\sum\limits_{j=M-1-M^{\frac{3}{4}}}^{M-1} \mu_j g(N,\mu_j)\;( (N+1)- (\frac{N}{N-1})^{M+2}(\frac{N-1}{N})^{j}(N-2) )\\
&=\sum\limits_{j=0}^{M^{\frac{3}{4}}} \mu_M g(N,\mu_{M-1-M^{\frac{3}{4}}})N^j\;( (N+1)- (\frac{N}{N-1})^{M^{\frac{3}{4}}+3}(\frac{N-1}{N})^{j}(N-2) )\;(1+o(\frac{1}{M^{\frac{1}{9}}}))\\
&= \mu_M g(N,\mu_{M-1-M^{\frac{3}{4}}})\;( \frac{N+1}{N-1}N^{M^{\frac{3}{4}}+1}- (\frac{N}{N-1})^{M^{\frac{3}{4}}+3}(N-1)^{M^{\frac{3}{4}}+1})\;(1+o(\frac{1}{M^{\frac{1}{9}}}))\\
&=\mu_M g(N,\mu_{M-1-M^{\frac{3}{4}}})N^{M^{\frac{3}{4}}+1}\;( \frac{N+1}{N-1}- (\frac{N}{N-1})^{2})\;(1+o(\frac{1}{M^{\frac{1}{9}}}))\\
&=- \mu_M g(N,\mu_{M-1-M^{\frac{3}{4}}})N^{M^{\frac{3}{4}}+1}\frac{1}{(N-1)^2} \;(1+o(\frac{1}{M^{\frac{1}{9}}})).
\end{split}
\end{align}
In what follows we use the fact that:
\begin{equation*}
e^{-C M^{\frac{1}{8}}}i^l g(i,\mu_{j})(1+o(\frac{1}{M^{\frac{1}{9}}}))\le g(i,\mu_{j+l})\le e^{C M^{\frac{1}{8}}}i^l g(i,\mu_{j})(1+o(\frac{1}{M^{\frac{1}{9}}})), 
\end{equation*}
where $C$ is a positive constant, depending on $N$. This yields

\begin{align}\label{Estimates_for_U_2}
\begin{split}
| U^{2} |\;
&=-\sum\limits_{j=0}^{M-1-M^{\frac{3}{4}}} \mu_j g(N,\mu_j)\;( (N+1)- (\frac{N}{N-1})^{M+2}(\frac{N-1}{N})^{j}(N-2) )\\
&\le -\sum\limits_{j=0}^{ M-1-M^{\frac{3}{4}} } e^{CM^{\frac{1}{8}}}\mu_M g(N,\mu_0)N^j\;( (N+1)- (\frac{N}{N-1})^{M+2}(\frac{N-1}{N})^{j}(N-2) )\;(1+o(\frac{1}{M^{\frac{1}{9}}}))\\
&= -e^{CM^{\frac{1}{8}}}\mu_M g(N,\mu_0)\;( \frac{N+1}{N-1}N^{M-M^{\frac{3}{4}}}- (\frac{N}{N-1})^{M+2}(N-1)^{M-M^{\frac{3}{4}}})\;(1+o(\frac{1}{M^{\frac{1}{9}}}))\\
&=e^{CM^{\frac{1}{8}}}\mu_M g(N,\mu_0)N^{M-M^{\frac{3}{4}}}\;( \frac{N+1}{N-1}- (\frac{N}{N-1})^{M^{\frac{3}{4}}+2})\;(1+o(\frac{1}{M^{\frac{1}{9}}}))
\\
&\le e^{2CM^{\frac{1}{8}}} \mu_M g(N,\mu_{M-1-M^{\frac{3}{4}}})( \frac{N+1}{N-1}- (\frac{N}{N-1})^{M^{\frac{3}{4}}+2})\;(1+o(\frac{1}{M^{\frac{1}{9}}})).
\end{split}
\end{align}
Eqns (\ref{Estimates_for_U_1}) and (\ref{Estimates_for_U_2}) show that $U^{1}+ U^{2}= U^{1}(1+ o(\frac{1}{M^{\frac{1}{9}}}))$.

By using \eqref{difference_of_taus} and \eqref{Estimates_for_U_1}, we get $$ |{\tau_{1}}^{(N)}-{\tau_{2}}^{(N)}|= o(\frac{1}{M^{\frac{1}{9}}}) \times \mu_M g(N,\mu_M)\times \frac{1}{(N-1)^2}(1+o(\frac{1}{M^{\frac{1}{9}}})).$$
Also $|R_\epsilon(c_N|c_{N})| \ge R_\epsilon((N,\mu _M)|c_{N})$. A direct computation shows that $R_\epsilon((N,\mu _M)|c_{N})=  \frac{\mu_M g(N,\mu_M)}{\epsilon(N-1)^2}$ $ (1+o(\frac{1}{M^{\frac{1}{9}}})),\; R_\epsilon(c_{N}|(N,\mu _M))=0, \; \text{and}\; R_\epsilon((N,\mu _M)|(N,\mu _M))= (1+o(\frac{1}{M^{\frac{1}{9}}}))\times -3 \mu_M$.
This completes the proof of  Theorem \ref{first_theorem_last_section}.
\end{proof}

\begin{remark}\label{remark_4_2_1}
For $N=2$, the above proof does not work, but a direct computation will yield the same result.
\end{remark}

\begin{thm}\label{second_theorem_last_section}
For $N \ge N' \ge 2$, given $\max\Big[R_\epsilon(c_N'|c_{N'+1}), 
R_\epsilon(c_{N'+1}|c_{N'})\; \Big]=o(\frac{1}{M^{\frac{1}{9}}})\;
|R_\epsilon(c_{N'}|c_{N'})| $, we have 
$$\max \Big[\;R_\epsilon(c_{N'}|c_{N'-1}), R_\epsilon(c_{N'-1}|c_{N'})\; \Big] 
= o(\frac{1}{M^{\frac{1}{9}}})\;\times
\min\Big[\;|R_\epsilon(c_N'|c_{N'})|,\; |R_\epsilon(c_{N'-1}|c_{N'-1})|\;\Big]$$. 
\end{thm}
\begin{proof}
Similar to proof of Theorem \ref{first_theorem_last_section}.
\end{proof}
\end{subsection}
\end{section}

\begin{section}{Closing comments}
\begin{enumerate}
\item In Section \ref{construc_multiple_maintainance_states}, for any $r>0$ (possibly depending on the parameter $M$), if we replace the family of operators $A^{i}$ by ${A^{i}}_{r}$, defined as
$$
{A^{i}}_{r} (\mu'| \mu) = r \times A^{i} (\mu'| \mu),\; \forall i \in X,\mu',\mu \in Z.
$$
We then carry out the construction by using ${A^{i}}_{r}$ instead of $A^{i}$, all the ensuing 
results would still hold true. Multiplying everything by $r$ represents a universal and uniform 
speeding up of the environment Markov chain. One might think that the asymptotic results 
emerging as $M \to \infty$ only hold because the environment process is being equipped with a 
larger number of possible states between $\frac{1}{2}$ and $\frac{3}{2}$, without being 
given the  adequate speed to run through them, to compensate for this it needs $c$ to make 
long jumps. However, this view is dispelled by the argument above.

\item The use of multiple maintenance states suggests an interesting direction. As was noted  
in Section \ref{general_construction}, the transition state $c$ has been introduced 
because we could \textbf{not} choose $\tau_{\epsilon} (x,z) > 0$ such that 
\begin{align}{\label{closing_remarks_1}}
\begin{split}
&\sum\limits_{(x,z) \in X \times Z} (-\tau_{\epsilon} (x_i,z_j))^{\mathbbm{1}_{z_j}(z) 
\mathbbm{1}_{x_i}(x)}
\left|A^{x}(z_j|z)Q^{z}(x_i|x)\right| g(x,z) = 0, \\
&\sum\limits_{(x',z') \in X \times Z} (-\tau_{\epsilon} (x_i,z_j))^{\mathbbm{1}_{z_j}(z') 
\mathbbm{1}_{x_i}(x')} 
\left|A^{x_i}(z'|z_j)Q^{z_j}(x'|x_i) \right|= 0.
\end{split}
\end{align}
Instead we selected $\tau_{\epsilon} (x,z) > 0$ to satisfy (\ref{choice1}) and (\ref{choice2}).
This forced us into using $c$ to accommodate (\textbf{A}) and (\textbf{B}) of Section 2.1. Now, assume that 
the natural space $X \times Z$ can be partitioned into parts $\{B_{h}\}_{h=1}^{H}$ so  
that there exists $\tau_{\epsilon} (x,z) > 0$, satisfying \ $\forall$ \ $B_{h}$, the property 
that
\begin{align}\label{closing_remarks_2}
\begin{split}
&\sum\limits_{(x_i,z_j) \in B_{h}} \sum\limits_{(x,z) \in X \times Z} (-\tau_{\epsilon} (x_i,z_j))^{\mathbbm{1}_{z_j}(z) \mathbbm{1}_{x_i}(x)}\left| A^{x}(z_j|z)Q^{z}(x_i|x)\right| g(x,z) = 0 \nonumber\\
&\sum\limits_{(x_i,z_j) \in B_{h}} \sum\limits_{(x',z') \in X \times Z} (-\tau_{\epsilon} (x_i,z_j))^{\mathbbm{1}_{z_j}(z') \mathbbm{1}_{x_i}(x')}\left| A^{x_i}(z'|z_j)Q^{z_j}(x'|x_i)\right| = 0.
\end{split}
\end{align}
In this case we can use a collection of maintenance sets $\{c_{h}\}_{h=1}^{H}$ such that 
each $c_h$ interacts only with the partition element $B_h$, and there is no interaction 
between $c_{h_1}$ and $c_{h_2}$, $\forall$ $1 \le c_{h_1}\; <\;c_{h_2} \le H$. If we can 
choose a partition where the size of the parts is small, we get a satisfying construction. 
Whenever a maintenance state $c_h$ is attained, the `repair' is not drastic, as the 
process returns to a state close to the state it occupied immediately before it jumped to  $c_h$  (in the same $B_h$).
 
With this in mind, it could be considered that a pair of matrix families $\{A^x\},\{Q^z\}$ 
for which such a partition of the natural space can be realized, represents a  
\enquote{nice} form of interaction between the environment and basic states. In fact, these 
constructions can be generalized to the situation where $X$ and $Z$ are compact metric spaces.  
This will be discussed in forthcoming papers. Here the notion of partition into 
\enquote{small} parts can be formally understood as the requirement that each element of the partition has a 
small diameter.
\end{enumerate}

\end{section}
\begin{section}{Acknowledgment}
I would like to express my deepest gratitude to Prof. Manfred Denker, Prof.  Yuri Suhov, Dr. Anna Levina and Prof. Guodong Pang for their insight and guidance.
\end{section}
\bibliographystyle{plain}
\bibliography{bibf}
\end{document}